\documentclass[12pt, reqno]{amsart}
\usepackage{hyperref}
\usepackage{amssymb,amsmath,graphicx,amsthm}
\def\simge{\underset\sim>}

\def\T{\text}

\def\1#1{\overline{#1}}

\def\2#1{\widetilde{#1}}

\def\3#1{\widehat{#1}}

\def\4#1{\mathbb{#1}}

\def\5#1{\frak{#1}}

\def\6#1{{\mathcal{#1}}}

\def\C{{\4C}}

\def\R{{\4R}}

\def\T{\text}
\newcommand{\Om}{\Omega}

\newcommand{\no}[1]{\|{#1}\|}

\def\R{{\Bbb R}}

\def\C{{\Bbb C}}

\def\la{\langle}
\def\ra{\rangle}
\def\di{\partial}
\def\dib{\bar\partial}
\def\Label#1{\label{#1}}

\def\simge{\underset\sim>}

\def\T{\text}

\def\1#1{\overline{#1}}

\def\2#1{\widetilde{#1}}

\def\3#1{\widehat{#1}}

\def\4#1{\mathbb{#1}}

\def\5#1{\frak{#1}}

\def\6#1{{\mathcal{#1}}}

\def\C{{\4C}}

\def\R{{\4R}}

\numberwithin{equation}{section}

\def\T{\text}

\theoremstyle{plain}

\newtheorem{theorem}{Theorem}[section]

\newtheorem{corollary}[theorem]{Corollary}

\newtheorem{lemma}[theorem]{Lemma}

\theoremstyle{definition}

\theoremstyle{remark}

\newtheorem{remark}[theorem]{Remark}

\begin{document}

\title[Kobayashi metric near a point of infinite type]{Boundary behavior of the Kobayashi metric near a point of infinite type}
  
  \author[T.V. Khanh]{Tran Vu Khanh}   
\begin{abstract}Under a potential-theoretical hypothesis named $f$-Property with $f$ satisfying $\displaystyle\int_t^\infty \dfrac{da}{a f(a)}<\infty$, we show that the Kobayashi metric $K(z,X)$ on a weakly pseudoconvex domain $\Om$, satisfies the estimate $K(z,X)\ge Cg(\delta_\Om(x)^{-1})|X|$ for any $X\in T^{1,0}\Om$ where $(g(t))^{-1}$ denotes the above integral and $\delta_\Om(z)$ is the distance from $z$ to $b\Om$.  
\\[2mm] {\it AMS Mathematics  Subject Classification (2000)}: Primary 32F45, 32H35
\\[1mm] {\it Key words and phrases}: Kobayashi metric, proper holomorphic map, finite and infinite type
\end{abstract}

\maketitle\tableofcontents 
\section{Introduction}
\Label{d1} Let $\Om$ be a pseudoconvex domain in $\C^n$ and $z_o$ be a boundary point. For a smooth monontonic
increasing function $f :[1+\infty)\to[1,+\infty)$ with $f (t)\le t^{\frac{1}{2}}$, we say that $\Om$ has the $f$-Property at $z_o$ if there exist a neigborhood $U$ of $z_o$ and a family of functions $\{\phi_\delta\}$ such that
\begin{enumerate}
  \item [(i)] $\phi_\delta$ are plurisubharmonic and $C^2$ on $U$ and $-1\le \phi_\delta \le0$;
  \item[(ii)] $\di\dib \phi_\delta\simge f(\delta^{-1})^2Id$ and $|D\phi_\delta|\lesssim  \delta^{-1}$ for any  $z\in U\cap \{z\in \Om:-\delta<r(z)<0\}$, where $r$ is a defining function of $\Om$.
  \end{enumerate} 
  Here and in what follows, $\lesssim$ and $\simge$ denote inequality up to a positive constant. Morever, we will use $\approx$ for the combination of $\lesssim$ and $\simge$. \\

In the joint work with G. Zampieri \cite{KZ10}, we show that the $f$-Property implies an $f$-estimate for the $\dib$-Neumann problem. In another paper \cite{KZ12}, we prove that an  $f$-estimate with $\dfrac{f}{\log }\to \infty$ at $\infty$  implies that the Bergman metric has a lower bound with the rate $g(t)=\dfrac{f}{\log}(t^{1-\eta})$ for $\eta>0$. The ideas leading to these results follow by Kohn \cite{Koh02}, Catlin \cite{Cat83, Cat87} and McNeal \cite{McN92a}. Combining the two results above, we obtain

\begin{theorem}\Label{B}Let $\Om$ be a pseudoconvex domain in $\C^n$ with $C^\infty$-smooth boundary and $z_o$ a point in the boundary $b\Om$. Assume that the $f$-Property holds at $z_o$ with $\dfrac{f}{\log }\nearrow \infty$ for $t\to\infty$. Then for any $\eta>0$ there is a neigborhood $U_\eta$ of $z_o$ and  a constant $C_\eta$ such that the Bergman metric $B$ of $\Om$ satisfies
\begin{eqnarray}\Label{b1}
B(z,X)\ge C_\eta \dfrac{f}{\log}(\delta^{-1+\eta}_\Om(z))|X|
\end{eqnarray}
for any $z\in U_\eta\cap\Om$ and $X\in T_z^{1,0}\C^n$.
\end{theorem}
 The purpose of this paper is to prove a  result similar to Theorem \ref{B} for the Kobayashi metric. Let us recall the definition of the Kobayashi metric. \\

 Let $\Om$ be a pseudoconvex domain in $\C^n$; the function $K: T^{1,0}\Om\to \R$ on the holomorphic tangent bundle, given by 
\begin{eqnarray}
\begin{split}
K(z,X)=&\inf\{ \alpha>0| \exists g: \Delta\to \Om   \T{~holomorphic with~} g(0)=z, g'(0)=\alpha^{-1}X\}\\
=& \inf\{r^{-1}|  \exists g: \Delta_r\to \Om   \T{~holomorphic with~} g(0)=z, g'(0)=X\},
 \end{split}
\end{eqnarray}
is called the Kobayashi metric of $\Om$.  Here $\Delta$ denotes the unit disc and $\Delta_r$ the disc in $\C$ centered at $0$ with radius $r$. \\

Our main result is the following 
\begin{theorem}\Label{t1}Let $\Om$ be a pseudoconvex domain in $\C^n$ with $C^2$-smooth boundary $b\Om$ and $z_o$ be a boundary point. Assume that $\Om$ has the $f$-Property at $z_o$ with $f$ satisfying $\displaystyle\int_t^\infty \dfrac{da}{a f(a)}<\infty$ for some $t>1$, and denote by $(g(t))^{-1}$ the above, finite, integral. Then, there is a neighborhood $V$ of $z_o$  such that
\begin{eqnarray}\Label{K1}
K(z,X)\simge  g(\delta^{-1}_{\Om}(z))|X|
\end{eqnarray}
for any $z\in V\cap\Om$ and $X\in T^{1,0}_z\C^n$. 
\end{theorem}

We remark that Theorem~\ref{t1} may apply to domains of both finite  and infinite type; in the first case we take $f(t)=t^\epsilon,\, g(t)=t^\epsilon$, and in the second $f(t) = \log^{1+\epsilon}$, $g(t)=\log^\epsilon t$. \\

Comparing with Theorem~\ref{B}, we reduce the $C^\infty$-smoothness of the boundary and slightly strengthen the hypothesis of $f$ in the $f$-Property since $\displaystyle\int_t^\infty \dfrac{da}{a f(a)}<\infty$ is stronger than $\underset{a\to \infty}{\lim}\dfrac{f(a)}{\log a}=\infty$. Morever,  we obtain a larger size of the lower bound of the Koybayashi metric; for example, in the case $f=t^\epsilon$  we have $g=t^\epsilon$ instead of $g=t^{\epsilon-\eta}$ and, for $f=\log^{1+\epsilon}t$ we have $g=log^\epsilon t$ instead of $\log^\epsilon(t^{1-\eta})$. \\

Using the $f$-Property constructed by Catlin in \cite{Cat87, Cat89}, McNeal \cite{McN91b, McN92b}, Khanh-Zampieri \cite{KZ10}, Khanh \cite{Kha10}, we have the following
\begin{corollary}1) Let $\Om$ be a pseudoconvex domain of finite type $m$ in $\C^n$. Then \eqref{K1} holds for $g(t)=t^{\frac{1}{m}}$ if  $\Om$ satisfies at least one of the following conditions: $\Om$ is strongly pseudoconvex, or $\Om$ is convex, or $n=2$, or $\Om$ is decoupled. In any case, we have $g(t)=t^\epsilon$ with $\epsilon=m^{-n^2m^{n^2}}$.\\

2) Let $\Om$ be defined by
$\Om=\{z\in \C^n:\T{Im}z_n+\sum_{j=1}^{n-1}P_j(z_j)<0\},$
where $\Delta P_j(z_j)\simge \dfrac{\exp(-1/|x_j|^\alpha)}{x_j^2}$ or $\dfrac{\exp(-1/|y_j|^\alpha)}{y_j^2}$ wih $\alpha<1$. Then \eqref{K1} holds for $g(t)=\log^{\frac1\alpha-1}t$.
\end{corollary}

The lower bound of the Kobayashi metric is an important tool in the function theory of several complex variables and has been studied by many authors. In the following, we briefly review some significant, classical results.\\
 
When $\Om$ is  strongly pseudoconvex  or else it is pseudoconvex of finite type in $\C^2$ and  decoupled or convex in $\C^n$, then the size of the  Kobayashi metric  has been described by I. Graham \cite{Gra75a}, D. Catlin \cite{Cat89},  G. Herbort \cite{Her92} and L. Lee \cite{Lee08}. In these classes of domains, there exists a quantity $M(z,X)$  which satisfies the asymptotic formula 
$$\underset{z\to b\Om}{\lim}M(z,X)=\delta_\Om^{-1/m}(z)|X^\tau|+\delta_\Om^{-1}(z)|X^\nu|,$$
(where, $X^\tau$ and $X^\nu$ are the tangential and normal components of $X$ and $m$ is the type of the boundary), 
such that 
$$K(z, X)\approx M(z,X).$$

For a general pseudoconvex domain in $\C^n$, K. Diederich and J. E. Fornaess \cite{DF79} proved, by using Kohn's algorithm \cite{Koh79}, that there is a $\epsilon>0$ such that $K(z,X)\simge \delta(z)^{-\epsilon}|X|$  if $b\Om$ is real analytic  of finite type.  By using the method of Catlin in \cite{Cat87, Cat89}, S. Cho \cite{Cho92} improved the result  of \cite{DF79} for domains which are not necessarily real analytic.  However, in the case of infinite type we know very little except from the recent results by S. Lee \cite{Lee01} for the exponentially-flat infinite type. \\

Among other uses of the lower bound of the Kobayashi metric, we mention the continuous extendibility of proper holomorphic maps to the boundary of a domain of  general type. We refer readers to \cite{Hen73, BF78e, DF79, Ran78} for this problem on domains of finite type.

\begin{theorem}\Label{fHolder}Let $\Om$ and $\Om'$ be pseudoconvex domains. Let $\eta, 0<\eta\le 1$ be such that there is a $C^2$ defining function $r$ of $\Om$ with the property that $-(-r)^\eta$ is strictly plurisubharmonic on $\Om$. Assume that $\Om'$ has the $f$-Property with $f$ satisfying $\displaystyle\int_t^\infty \dfrac{(\ln a-\ln t) da}{a f(a)}<\infty$ for some $t>1$, and denote by $(\tilde f(t))^{-1}$ this finite integral. Then any proper holomorphic map $\Psi: \Om\to \Om'$ can be extended as a general H\"older continuous map $\hat\Psi:\bar\Om\to \bar\Om'$ with a rate $\tilde f(t^{\eta})$, that is, 
$$|\hat\Psi(z)-\hat\Psi(w)|\lesssim \tilde{f}(|z-w|^{-\eta})^{-1}$$
for any $z,w\in \bar\Om$. 
\end{theorem}
The paper is organized as follows. In section 2, using the $f$-Property, we construct the bumping functions.  By the  existence  of  suitable exhaustion functions, we obtain the plurisubharmonic peak functions having the good estimates. The lower bound of the Kobayashi metric follows from the estimates of the plurisubharmonic peak functions (cf. Section 3). In Section 4, we prove Theorem~\ref{fHolder}. \\


\section{The bumping function}
In this section, we construct the bumping functions, which might also be useful for other purposes. We will prove that, for any  boundary point $z_o$ on $b\Om$ which satisfies the $f$-Property, we can find a pseudoconvex hypersurface touching $\bar\Om$ exactly at  $w$ from the outside such that the distance from $z\in\Om$ to the new hypersurface is exactly controlled by the rate in$ |z-w|^{-1}$ of the reciprocal of the inverse of $g$.
\begin{theorem} \Label{bumpingfunct}
Let $\Om$ be pseudoconvex and $z_o$ be a boundary point.  Assume that $\Om$ has the $f$-Property at $z_o$ with $f$ satisfying $\displaystyle\int_t^\infty \dfrac{da}{a f(a)}<\infty$ for some $t>1$, and denote by $(g(t))^{-1}$ this finite integral. Then there is a neigborhood $U$ of $z_o$ and a real $C^2$ function $\rho$ on $U\times(U\cap b\Om)$ with the following properties:   
\begin{enumerate}
  \item $\rho(w,w)=0$.
  \item $\rho(z,w)\le -G(|z-w|)$ for any $(z,w)\in (U\cap \Om)\times (U\cap b\Om)$ where $G(\delta)=\left(g^*(\gamma\delta^{-1}\right)^{-1}$. Here, the supercript $^*$ denotes the inverse function and $\gamma>0$ sufficiently small. 
  \item $\rho(z,\pi(z))\gtrsim -\delta_\Om(z)$ for any $z\in U\cap \Om$ where $\pi(z)$ is the projection of $z$ to the boundary.
  \item For each fixed $w\in U\cap b\Om$, denote $S_w=\{z\in U: \rho(z,w)=0 \}$. One has:
  \begin{enumerate}
    \item $|D_z\rho(z,w)| \approx1$ everywhere on $S_w$.
     \item $S_w$ is pseudoconvex. In fact, one can choose $\rho$ such that $S_w$ is strongly pseudoconvex outside of $w$.
     \item $S_w$ touches $\bar\Om$ exactly at $w$ from outside.
\end{enumerate}
\end{enumerate}
\end{theorem}

 The proof is divided in four steps. In step 1, we show the equivalence of the $f$-Property between the pseudoconvex and the pseudoconcave side of a hypersurface. In step 2, we prove that there exists a single function with self-bounded gradient which has a lower bound $f(r^{-1}(z))$ for the Levi form.  In step 3, we estimate the function $G$.  The properties of bumping function is checked on step 4. \\ 
 
{\it Proof of Theorem~\ref{bumpingfunct}.}\\

{\bf Step 1.}  Since the hypersurface defined by each bumping function lies outside the original domain except from  one point and the $f$-Property takes place inside the domain,  we first show hat the $f$-Property still holds outside the domain. \\

 Without loss of generality, we can assume that the original point $z_o$ belongs to $U\cap b\Om$. We choose  special coordinates $z=(x,r)\in \R^{2n-1}\times \R$ at $z_o$. Assume that there is a family of functions $\phi_\delta$ which have properties (i) and (ii) in the first paragraph of Section 1. Define  $\tilde\phi_\delta(x,r):=\phi_\delta(x,r-\delta)$ for each $\delta>0$ and still call $\phi_\delta$ for $\tilde\phi_\delta$.  Then,  for each $\delta$, $\phi_\delta$ is $C^2$, plurisubharmonic, and satisfies on $U$ $-1\le\phi_\delta\le0$, $\di\dib \phi_\delta\simge f^2(\delta^{-1})Id$, and $|D\phi_\delta|\lesssim \delta^{-1}$ on $-\delta<r-\delta<0$ or $0<r<\delta$.\\

{\bf Step 2.} In this step we build a single function which has self-bounded gradient and has  lower bound $f(r^{-1}(z))$ for the Levi form.
 \begin{lemma}\Label{A3} Assume that $\Om$ enjoys the $f$-Property at $z_0$. Then there is a single function $\Phi$ and constants $c,C>0$ such that
\begin{enumerate}
  \item $-1\lesssim \Phi\le 0$
  \item $\di\dib\Phi(X,X)\ge-C(\dfrac{1}{r}|\di\dib r(X,\bar X)|+\dfrac{1}{r^2}|Xr|^2)+\dfrac{1}{8}|X\Phi|^2+ cf^2(\dfrac{1}{r})|X|^2$
  \item $|D\Phi|\lesssim \dfrac{1}{r}$
\end{enumerate}
for any $z\in U\setminus\bar\Om$.
\end{lemma}
{\it Proof. } Let $\chi$ be a cut-off function such that 
$\chi(t)=
\begin{cases}
0 &\T{~if~} t\le \dfrac{1}{4} ~~\T{or}~~t\ge 2.\\
1&\T{~if~}  \dfrac{1}{2}\le t\le 1.
\end{cases}$ We also suppose that $|\dot\chi|$, $|\ddot\chi|$ and $\dfrac{\dot\chi^2}{\chi}$ are bounded.  
Define 
\begin{eqnarray}\Label{a1}
\Phi(z):=\sum_{j=1}^\infty\left(\exp(\phi_{2^{-j}}(z))-1\right)\chi(2^jr(z)).
\end{eqnarray}
Denote $S_\delta:=\{z\in U| 0<r(z)<\delta\}$. Let $z\in U\setminus\bar\Om=\underset{j=1}{\overset{\infty}{\cup}} S_{2^{-k}}\setminus S_{2^{-(k+1)}}$, then there is an integer $k$ such that 
\begin{eqnarray}\Label{2.1a}
z\in S_{2^{-k}}\setminus S_{2^{-(k+1)}}=\{z\in U: 2^{-k-1}\le r(z)<2^{-k}\}.
\end{eqnarray}
We notice that $\chi(2^{j}r(z))=0$ if $j<k-1$ or $j>k+1$, and $\chi(2^kr(z))=1$ for any $z\in S_{2^{-k}}\setminus S_{2^{-(k+1)}}$. Hence, \eqref{a1} can be rewritten that 
$$\Phi(z)=\sum_{j=k-1}^{k+1}\left(\exp(\phi_{2^{-j}}(z))-1\right)\chi(2^jr(z)).$$ 
This proves (i). We observe that 
\begin{eqnarray}\Label{2.2}
\begin{split}
\di\dib\left(\left(e^{\phi_{2^{-j}}}-1\right)\chi(2^jr)\right)(X,\bar X)=
&\left(\di\dib\phi_{2^{-j}}(X,\bar X)+|X\phi_{2^{-j}}|^2\right ) e^{\phi_{2^{-j}}}\chi\\
&+2^{j+1}\T{Re}\la X\phi_{2^{-j}}, \overline{Xr}\ra e^{\phi_{2^{-j}}}\dot\chi\\
&+\left(2^j\di\dib r(X,\bar X)\dot\chi+2^{2j}|Xr|^2\ddot\chi\right)(e^{\phi_{2^{-j}}}-1)\\
\ge&\left( \di\dib\phi_{2^{-j}}(X,\bar X)+\dfrac{1}{2}|X\phi_{2^{-j}}|^2\right)e^{\phi_{2^{-j}}}\chi\\
&-2^j\dot\chi(1-e^{\phi_{2^{-j}}})|\di\dib r(X,\bar X)|\\
&-2^{2j}(|\ddot\chi|(1-e^{\phi_{2^{-j}}})+2\frac{\dot\chi^2}{\chi}e^{\phi_{2^{-j}}})|Xr|^2.
\end{split}
\end{eqnarray}
Here, we use the Cauchy-Schwartz inequality for the second line of  \eqref{2.2}, that is,
$$\left|2^{j+1}\T{Re}\la X\phi_{2^{-j}} \overline{Xr}\ra  e^{\phi_{2^{-j}}}\dot\chi\right| \le \frac{1}{2}|X\phi_{2^{-j}}|^2e^{\phi_{2^{-j}}}\chi+2^{2j+1}|Xr|^2\frac{\dot\chi^2}{\chi}.$$
Moreover, we also observe that
\begin{eqnarray}\Label{2.3}
\begin{split}
|X\Phi(z)|^2=&|\sum_{j=k-1}^{k+1}X(\phi_{2^{-j}})e^{\phi_{2^{-j}}}\chi+2^j(e^{\phi_{2^{-j}}}-1)X(r)\dot\chi(2^jr)|^2\\
\le & 4\sum_{j=k-1}^{k+1}|X\phi_{2^{-j}}|^2e^{\phi_{2^{-j}}}\chi(2^jr)+2^{2k+2}(1-e^{-1})^2|Xr|^2\left(\frac{1}{4}\dot\chi^2(2^{k-1})+4\dot\chi^2(2^{k+1})\right)
\end{split}
\end{eqnarray}
Combining \eqref{2.2} and \eqref{2.3}, we obtain
\begin{eqnarray}\Label{2.4}
\begin{split}
\di\dib\Phi(X,\bar X)\ge&e^{-1}\di\dib\phi_{2^{-k}}(X,\bar X)+\dfrac{1}{8}|X\Phi|^2-C(2^k|\di\dib r(X,\bar X)|+2^{2k}|Xr|^2)\\
\ge&cf^2(2^k)|X|^2+\dfrac{1}{8}|X\Phi|^2-C(2^k|\di\dib r(X,\bar X)|+2^{2k}|Xr|^2)\\
\end{split}
\end{eqnarray}
for any $z\in S_{2^{-k}}\setminus S_{2^{-(k+1)}}$. From \eqref{2.3}, we also obtain $|D\Phi|\lesssim 2^k$ $z\in S_{2^{-k}}\setminus S_{2^{-(k+1)}}$ since $|D\phi_\delta|\lesssim \delta^{-1}$ and $|Dr|\lesssim 1$. This completes the proof of (ii) and (iii).\\

$\hfill\Box$

{\bf Step 3.} We recall that $g(t)=\left(\displaystyle\int_t^\infty \dfrac{da}{a f(a)}\right)^{-1}=\left(\displaystyle\int_0^{t^{-1}}\dfrac{da}{a f(a^{-1})}\right)^{-1}$ for any $t>1$. Then, it is easy to check that $g$ is increasing, $g\to \infty$ at $\infty$, and $g\le f$ on $(1,+\infty)$.  We define 
$$G(\delta):=g^*\left(((\gamma\delta)^{-1}\right)^{-1}$$ where $\gamma>0$ is a constant to be chosen later. We also notice that $G$ is an increasing function, and $G(0)=0$.\\
{\it  Claim: For $\delta>0$, we have
\begin{enumerate} 
\item $\dfrac{\dot G(\delta)}{G(\delta)}=\gamma f(G^{-1}(\delta))$;
\item $G(\delta)\ddot G(\delta)\le \dot G^2(\delta)$;
\item $\dfrac{G(\delta)}{\delta}\le \dot G(\delta)$.
\end{enumerate}}
{\it Proof of the Claim.} 
 By the definition of $G$ and $g$, we have 
\begin{eqnarray}\Label{2.6}
g(G(\delta)^{-1}=(\gamma\delta)^{-1}\quad \T{ or }\quad \displaystyle\int_0^{G(\delta)}\dfrac{da}{a f(a^{-1})}=\gamma \delta.
\end{eqnarray}
Taking the derivative with respect to $\delta$ in the second equation of \eqref{2.6}, we prove the first claim, that is, 
\begin{eqnarray}\Label{2.7}
\dfrac{\dot G(\delta)}{G(\delta)}=\gamma f(G^{-1}(\delta)).
\end{eqnarray}
Taking again  the derivative with respect to $\delta$ in \eqref{2.7}, and observing that
\begin{eqnarray}\Label{2.8}
-\frac{\gamma\dot G(\delta)\dot f(G^{-1}(\delta))}{G^2(\delta)}=\dfrac{G(\delta)\ddot G(\delta)-\dot G^2(\delta)}{G^2(\delta)},
\end{eqnarray}
(since $G$ and $f$ are increasing functions) we get the proof of the second claim. Moreover, since $f\ge g$, then $G^{-1}(\delta)=g^*\left((\gamma\delta)^{-1}\right)\ge f^*\left((\gamma\delta^{-1})\right)$. From \eqref{2.7} we then get $\dfrac{\dot G(\delta)}{G(\delta)}\ge \gamma f(f^*((\gamma\delta)^{-1}))=\delta^{-1}$. The proof of the Claim is complete.\\

$\hfill\Box$

{\bf Step 4.} We define 
$$\rho(z,w)=r(z)+G(|z-w|)\left(-1+\epsilon\Phi(z)\right)$$
where $\epsilon>0$ will be chosen later. \\

Let $S_w= \{z\in U| \rho(z,w)=0\}$ be a hypersurface defined by $\rho(z,w)=0$ where $w$ is fixed. 
We will prove that $\rho$ satisfies the following properties:
\begin{enumerate}
\item[(i)] $\rho(w,w) = 0$ for any $w\in b\Om$.
\item[(ii)] $\rho(z,w)\le -G(|z-w|)$ for $z\in U\cap \Om$ and $w\in U\cap b\Om$.
\item[(iii)] $\rho(z,\pi(z)) \gtrsim - r(z)$ for $z\in U\cap \Om$, where $\pi(z)$ is the projection of $z$ to the boundary.
\item[(v)] $S_w$ is pseudoconvex.
\item[(iv)] $|D_z \rho(z,w)|\approx 1$ on $S_w$.
\end{enumerate}
Now, (i) is obvious.  
Since $\Phi$ is negative and bounded,  we first choose $\epsilon$ so small  that $-2\le -1+\epsilon\Phi\le -1$.
For $z\in U\cap \Om$, we have $r(z)<0$, and $|r(z)|\ge G(|r(z)|)$,  hence (ii) and (iii) follow.  \\

We estimate the Levi form of $\rho$ with respect to $z$,
\begin{eqnarray}\Label{phi}
\begin{split}
\di_z\dib_z\rho(z,w)(X,X)=&\di\dib r(z)(X,X)+\big(\dfrac{\dot{G}(|z-w|)}{|z-w|}+\ddot{G}(|z-w|)\big)(-1+\epsilon\Phi(z))|X|^2\\
&+2\epsilon\T{Re}\la XG(|z-w|),X\Phi(z)\ra+\epsilon G(|z-w|)\di\dib \Phi(z)(X,X)\\
\ge&\di\dib r(z)(X,X)-2\big(\dfrac{\dot{G}(|z-w|)}{|z-w|}+\ddot{G}(|z-w|)+8\epsilon\dfrac{\dot{G}^2(|z-w|)}{G(|z-w|)} \big)|X|^2\\
&+\epsilon G(|z-w|)\big(\di\dib \Phi(z)(X,X)-\frac{1}{16}|X\Phi(z)|^2\big)\\
\ge&\di\dib r(z)(X,X)-\epsilon C G(|z-w|)\left(\dfrac{1}{r}|\di\dib r(X,\bar X)|+\dfrac{1}{r^2}|Xr|^2\right)\\
&+\epsilon c G(|z-w|)f^2(\dfrac{1}{r})|X|^2+\dfrac{\epsilon }{16}G(|z-w|)|X\Phi|^2 \\
&-2\big(\dfrac{\dot{G}(|z-w|)}{|z-w|}+\ddot{G}(|z-w|)+8\epsilon\dfrac{\dot{G}^2(|z-w|)}{G(|z-w|)} \big)|X|^2\\
\end{split}
\end{eqnarray} 
Here, the first inequality follows from Cauchy-Schwartz inequality as for the second line; the last inequality follows from Lemma~\ref{A3}(ii).\\

Now we consider $z\in (S_w\cap U)\setminus{w} \subset U\setminus \bar\Om$, that is, $r(z)=G(|z-w|)\big(1-\epsilon\Phi(z)\big)$. By the choice of $\epsilon$, we obtain 
\begin{eqnarray}\Label{c2.10}
G(|z-w|)\le| r(z)|\le 2G(|z-w).
\end{eqnarray}
Thus the inequality of \eqref{phi} continues as
\begin{eqnarray}\Label{b2}
\begin{split}
\di_z\dib_z\rho(z,w)(X,X)\ge&\di\dib r(z)(X,X)-2\epsilon C|\di\dib r(X,X)|-\frac{4C\epsilon }{G(|z-w|)}|Xr|^2\\
&+\frac{\epsilon}{16}G(|z-w|)|X\Phi|^2\\
&+\left( c\epsilon G(|z-w|)f^2\left(\frac{1}{G(|z-w|}\right)-(4+16\epsilon)\dfrac{\dot{G}^2(|z-w|)}{G^2(|z-w|)} \right)|X|^2\\
\end{split}
\end{eqnarray} 
Here the last line follows from Claim (2) and (3).\\

Choose $\epsilon$ small such that $2\epsilon C\le 1$; the first line of \eqref{b2} can be estimated as follows
\begin{eqnarray}\Label{2.11}
\begin{split}
&\di\dib r(z)(X,X)-2\epsilon C |\di\dib r(X,X)|-\frac{4\epsilon C}{G(|z-w|)}|Xr|^2\\
&\ge \di\dib r(z)(X,X)-|\di\dib r(X,X)| -\frac{2}{G(|z-w|)}|Xr|^2\\
&\ge C_1|X| |Xr|-\frac{2}{G(|z-w|)}|Xr|^2\\
&\ge-\frac{C_1^2}{4}G(|z-w|)|X|^2-\frac{3}{G(|z-w|)}|Xr|^2\\
\end{split}
\end{eqnarray} 
For $X\in T^{1,0}S_w$, that is, $X\rho=0$. This implies
 $$Xr=(1-\epsilon \Phi)XG(|z-w|)-\epsilon G(|z-w|)X\Phi,$$
and hence, 
$$
|Xr|^2\le 8\dot G^2(|z-w|)|X|^2+2\epsilon^2 G^2(|z-w|)|X\Phi|^2.
$$
The inequality \eqref{2.11} continues as 
\begin{eqnarray}\Label{2.12}
\begin{split}
&\ge-C_2 \frac{\dot G^2(|z-w|)}{G(|z-w|)}|X|^2-6\epsilon^2 G(|z-w|)|X\Phi|^2.
\end{split}
\end{eqnarray} 
Combining \eqref{b2} and \eqref{2.12}, we obtain
\begin{eqnarray}\Label{2.13}
\begin{split}
\di\dib\rho(X,\bar X)\ge& \left(\dfrac{\epsilon}{16}-6\epsilon^2\right)G(|z-w|)|X\Phi|^2\\
&+c\epsilon G(|z-w|)\left(f^2\left(\frac{1}{G(|z-w|)}\right)-C_3 \frac{\dot G^2(|z-w|)}{G^2(|z-w|)}\right)|X|^2
\end{split}
\end{eqnarray} 
Again, choose $\epsilon$  such that $\dfrac{\epsilon}{16}-6\epsilon^2\ge 0$; then the term in left hand side of the first line of \eqref{2.12} can be disregarded. Using Claim (1) with $\gamma>0$  small enough, we obtain that the term in the second line is positive. We conclude that $\di\dib\rho(X,X)\ge 0$ on $S_w$ for any $X\in T^{1,0}S_w$. The proof of property (iv) is complete. \\

For any $z\in (S_w\cap U)\setminus {w}$, we have
\begin{eqnarray}
\begin{split}
\left|D\left(G(|z-w|)(-1+\epsilon\Phi(z))\right)\right|\le& 2\dot G(|z-w|)+\epsilon G(|z-w|)|D\Phi(z)|\\
\le& 2\eta G(|z-w|)f(G^{-1}(|z-w|))+\epsilon \frac{G(|z-w|)}{r(z)}\\
\le& 2\eta+\epsilon\\
\end{split}
\end{eqnarray}  
 where, the second inequality follows from Claim (1) and Lemma~\ref{A3}.(3), the third inequality follows from the hypothesis that $f(t)\le t$ and \eqref{c2.10}. Since $|Dr|\approx 1$, then for $\epsilon$ and $\eta$ small enough, we obtain $|D\rho|\approx 1$. That is the proof of property (iv).

The proof of Theorem \ref{bumpingfunct} is complete.

$\hfill\Box$

\section{Proof of Theorem~\ref{t1}}
The proof of Theorem~\ref{t1} follows immediately from Theorem~\ref{pshpeak} and Theorem~\ref{lowerbound}  below. 
 Theorem~\ref{pshpeak} consists in the  construction of plurisubharmonic peak functions with good estimates. This is a consequence of the construction of bumping functions in last section. More precisely, we obtain the following  
\begin{theorem}\Label{pshpeak}Assume that there exists a family of bumping functions on a local path $V$ of the boundary as in the conclusion of Theorem \ref{bumpingfunct}.  Fix $0< \eta<1$; then for any $w\in V\cap b\Om$ there is a plurisuhharmonic function $\psi_w$ on $V\setminus \{w\}$ verifying
 \begin{enumerate}
   \item $|\psi_w(z)-\psi_w(z')|\lesssim |z-z'|^\eta$
      \item $\psi_w(z)\le -G^\eta(|z-w|)$
   \item $\psi_{\pi(z)}(z)\gtrsim-\delta_\Om(z)^{\eta}$
 \end{enumerate}
for any $z$ and $z'$ in $V\cap\bar\Om$.
\end{theorem}
\begin{remark} The construction of the plurisubharmonic peak functions on a pseudoconvex domain of finite type in $\C^2$ and a convex domain of finite type in $\C^n$ has been obtained by J. E. Fornaess and N. Sibony in \cite{FS89}.
\end{remark}
{\it Proof of Theorem~\ref{pshpeak}.} Using the argument in Section 3 and 4 of Diederich-Fornaess \cite{DF79}, we obtain that for any $\eta>0$, there exist an open neighborhood $V\subset U$, and a constant $L>0$, such that 
$$\psi_w(z)=-\left(-\rho(z,w) e^{L|z-w|^2}\right)^{\eta},$$
is a plurisubharmonic function on $V\cap \{z\in U:\rho(z,w)<0\}$. By the properties of $\rho$, we can check that $\psi_w$ satisfies (1), (2) and (3). That is the proof of Theorem~\ref{pshpeak}.\\
  
$\hfill\Box$

Now, we prove the lower bound for the Kobayashi metric  by using the plurisubharmonic peak function. We state the theorem in a more general setting

\begin{theorem}\Label{lowerbound}
Let $\Om$ be a pseudoconvex domain in $\C^n$, $z_o$ be a given boundary  point, $F_1$ and $F_2$ are postive functions such that $F_1$ is increasing and convex. Assume that there is a neighborhood $V$ of $z_o$ such that for each $w\in V\cap b\Om$, there is a plurisubharmonic function $\psi_w$ such that 
\begin{enumerate}
  \item[i)] $\psi_w(z)\le -F_1(|z-w|)$ 
    \item[ii)] $\psi_{\pi(z)}(z)\ge -F_2(\delta_\Om(z))$ 
  \end{enumerate}
for $z\in U\cap\Om$. \\
Then $$K_\Om(z,X)\ge (F_1^*(F_2(\delta_\Om(z)))^{-1}|X|$$ for all $z\in V\cap \Om$, $X\in T_z^{1,0}{\C^n}$.
\end{theorem}
\begin{proof}
 We  fix now a point $z\in V\cap \Om$, put $w=\pi(z)$ and assume that $g=(g_1,\dots,g_n):\overline\Delta\to \Om$ is a holomorphic map of the closed unit disc into $\Om$ with $g(0)=z$.\\

By applying the mean value inequality to the subharmonic function $\psi_w(g(t))$ on $\overline\Delta$ we get
$$\psi_w(z)=\psi_w(g(0))\le \int_{0}^{1}\psi_w\circ g(e^{i2\pi\theta})d\theta$$
The hypothesis (ii) gives
\begin{eqnarray}\Label{ineq1}
 F_2(\delta(z)))\ge \int_{0}^1-\psi_w\circ g(e^{i2\pi\theta})d\theta 
\end{eqnarray}
We now use the hypothesis (i) of $\psi_w$, 
\begin{eqnarray}\Label{ineq2}
\begin{split}
\int^1_0-\psi_w\circ f(e^{i2\pi \theta})d\theta=&\int^1_0\big(-\psi_w\circ g(e^{i2\pi \theta})-F_1(|g_j(e^{i2\pi\theta})-w_j|)\big)d\theta\\
&+\int^1_0F_1\big(|g_j(e^{i2\pi\theta})-w_j|\big)d\theta\\
\ge &\int^1_0F_1\big(|g_j(e^{i2\pi\theta})-w_j|\big)d\theta.
\end{split}
\end{eqnarray}
Using the Jensen inequality for the increasing, convex function $F_1$, we get 
$$F_1( |g_j'(0)|)\le F_1 \Big(\int^1_0|g_j(e^{i2\pi\theta})-w_j|d\theta\Big)\le \int^1_0 F_1\big(|g_j(e^{i2\pi\theta})-w_j|\big)d\theta .$$
Combining the above inequality with \eqref{ineq1} and \eqref{ineq2}, we obtain
$$F_1(|g'_j(0)|)\le  F_2(\delta_\Om(z)).$$
An immediate consequence of this is 
$$|g'_j(0)|\le  F_1^*(F_2(\delta_\Om(z)).$$
By the definition of $K(z,X)$, we must have for all $X
\in T^{1,0}\C^n$
$$K(z,X)\ge  (F_1^*(F_2(\delta_\Om(z)))^{-1}|X|.$$

\end{proof}

\section{Application to proper holomorphic maps }
Let $\Om_1$ and $\Om_2$ be bounded domains in $\C^n$ with smooth boundary. Assume that $\Om_2$ is pseudoconvex of finite type at every boundary point. It is well-known that there is $\alpha>0$, such that every proper holomorphic map $\Psi: \Om_1\to \Om_2$ is H\"{o}lder continuous of order $\alpha$, in particular, $\Psi$ extends continuosly to $\bar\Om_1$. In this section, we prove a similar  result for domains of infinite type. For this purpose we give a suitable estimate for generalized  H\"older regularity.\\

Let $f$ be an increasing function such that $\underset{t\to+\infty}{\lim} f(t)=+\infty$. For $\Om\subset \C^n$, define the $f$-H\"older space on $\bar\Om$ by
$$\Lambda^{f}(\bar\Om)=\{u : \no{u}_{\infty}+\sup_{z,z+h\in \overline\Om}f(|h|^{-1}) \cdot |u(z+h)-u(z)|<\infty \}$$ 
and set 
$$\no{u}_{f}= \no{u}_{\infty}+\sup_{z,z+h\in \overline\Om}f(|h|^{-1})\cdot |u(z+h)-u(z)|. $$
Note that the $f$-H\"older space include the standard H\"older space $\Lambda_\alpha(\bar\Om)$ by taking $f(t) = t^{\alpha}$ (so that $f(|h|^{-1}) = |h|^{-\alpha}$) with $0<\alpha<1$. \\

Before proving Theorem~\ref{fHolder}, we need a generalization of the  Hardy-Littlewood Lemma. 
\begin{lemma} \Label{HL}Let $\Om$ be a bounded Lipschitz domain in $\R^N$ and let $\delta_{\Om}(x)$ denote the distance function from $x$ to the boundary of $\Om$. Let $G:\R^+\to\R^+$ be an increasing function such that $\dfrac{G(\delta)}{\delta}$ is decreasing and $\displaystyle\int_0^d\frac{G(\delta)}{\delta}d\delta<\infty$ for $d>0$ small enough.  Let $u\in C^1(\Om)$ satisfy
\begin{eqnarray}\Label{5.1}
\quad |\nabla u(x)|\lesssim \frac{G(\delta_{\Om}(x))}{\delta_{\Om}(x)}~~\T{~~for every ~~} x\in \Om.
\end{eqnarray}
Then $u\in \Lambda^f(\bar\Om)$  where $f(d^{-1})=\Big(\displaystyle\int_0^d\frac{G(\delta)}{\delta}d\delta\Big)^{-1}.$
\end{lemma}
The proof of this theorem can be found in \cite{Kha12}. 
\begin{remark} If $G(t)=t^\alpha$, Lemma \ref{HL} is the classical Hardy-Littlewood Lemma for a domain of finite type. 
\end{remark}
{\it Proof of Theorem~\ref{fHolder}}  Using Theorem~\ref{t1} for $\Om'$, the Schwarz-Pick lemma for the Kobayashi metric, and the upper bound of the Kobayashi metric, we obtain the  following estimate
\begin{eqnarray}
g\left(\delta_{\Om'}^{-1}(\Psi(z))\right)|\Psi'(z)X|\lesssim K_{\Om'}(\Psi(z), \Psi'(z)X)\le K_\Om(z,X)\lesssim \delta^{-1}_\Om(z)|X|
\end{eqnarray}
for any $z\in\Om$ and $X\in T^{1,0}\C^n$.
Moreover, by the fact that $-(-r)^\eta$ is strictly plurisubharmonic on $\Om$, one has $\delta_{\Om'}(\Psi(z))\lesssim \delta^\eta_\Om(z)$ for any $z\in\Om$  (Lemma 8 in \cite{DF79}). Therefore, 
$$|\Psi'(z)X|\lesssim \delta^{-1}_\Om(z) g^{-1}(\delta_\Om^{-\eta}(z))|X|$$
for any $z\in\Om$ and $X\in T^{1,0}\C^n$. Using  Lemma~\ref{HL}, $\Psi$ can be extended to a $h$-H\"older continuous map $\hat\Psi:\bar\Om\to\bar\Om'$ with the rate $h(t)$ defined by
\begin{eqnarray}
\begin{split}
(h(t))^{-1}:=&\int_0^{t^{-1}}\frac{d\delta}{\delta g(\delta^{-\eta})}=\frac{1}{\eta}\int_{t^\eta}^{\infty}\frac{db}{b g(b)}\\
=&\frac{1}{\eta}\int_{t\eta}^\infty\frac{1}{b} \left(\int_{b}^\infty \frac{da}{a f(a)}\right)db=\frac{1}{\eta}\int_{t^\eta}^\infty\frac{1}{af(a)} \left(\int_{t^\eta}^a \frac{db}{b}\right)da\\
=&\frac{1}{\eta}\int_{t^\eta}^\infty\frac{\ln a-\ln t^{\eta}}{af(a)}da=\frac{1}{\eta}(\tilde f(t^\eta))^{-1}.\\
\end{split}
\end{eqnarray}
 The proof of Theorem~\ref{fHolder} is complete.

$\hfill\Box$
\bibliographystyle{alpha}

\end{document}